\documentclass[12pt, a4paper]{amsart}
\usepackage{amssymb,fullpage}
\usepackage{amsmath,esint}
\usepackage[colorlinks, allcolors=RedViolet]{hyperref}
\usepackage{colortbl}
\usepackage[dvipsnames]{xcolor}

\newtheorem{theorem}[equation]{Theorem}
\newtheorem{lemma}[equation]{Lemma}

\newtheorem{corollary}[equation]{Corollary}
\theoremstyle{definition}

\theoremstyle{remark}
\newtheorem{remark}[equation]{Remark}

\numberwithin{equation}{section}

\newcommand{ \R }{ \mathbb{R} }

\newcommand{ \tz }{ \tilde{z} }

\newcommand{\loc}{{\operatorname{loc}}}
\renewcommand{\epsilon}{\varepsilon}
\renewcommand{\phi}{\varphi}
\renewcommand{\le}{\leqslant}
\renewcommand{\ge}{\geqslant}
\renewcommand{\leq}{\leqslant}

\makeatletter
\@namedef{subjclassname@2020}{\textup{2020} Mathematics Subject Classification}
\makeatother

\usepackage[comment]{paper_diening}

\begin{document}

\title{Parabolic weighted Sobolev-Poincar\'e type inequalities}

\author{Lars Diening}

\address{Lars Diening, University Bielefeld, Universit\"atsstrasse 25, 33615 Bielefeld, Germany}
\email{lars.diening@uni-bielefeld.de}

\author{Mikyoung Lee}
 \address{Mikyoung Lee, Department of Mathematics, Pusan National University, Busan 46241, Republic of Korea}
\email{\texttt{mikyounglee@pusan.ac.kr}}

\author{Jihoon Ok}
\address{Jihoon Ok, Department of Mathematics, Sogang University, Seoul 04107, Republic of Korea}
\email{\texttt{jihoonok@sogang.ac.kr}}

\thanks{}





\begin{abstract} We derive weighted Sobolev-Poincar\'e type inequalities in function spaces concerned with parabolic partial differential equations. We consider general weights depending on both space and time variables belonging to a Muckenhoupt class, so-called the parabolic $A_p$-class, where only the parabolic cubes are involved in the definition.
\end{abstract}

\maketitle
\section{\bf Introduction}
Sobolev and Poincar\'e type inequalities are fundamental tools investigating relevant Sobolev spaces and related partial differential equations (PDEs). For classical weighted Sobolev spaces they  have been studied for a long time in for instance
\cite{Bo88,CW85,CF85,Ch93,Ch01,DD08,FKS82,HK00,SW85}.  In particular,  Fabes, Kenig and Serapioni  
\cite{FKS82}  obtained the  weighted Sobolev-Poincar\'e inequality 
\begin{equation}\label{classical}
\left(\frac{1}{w(B_r)}\int_{B_r} |f-(f)_{B_r}|^{pk}w\,dx \right)^{\frac{1}{pk}}  \le  c\,r \left(\frac{1}{w(B_r)}\int_{B_r}|Df|^p
w\,dx\right)^{\frac{1}{p}}, 
\end{equation}
and, using this,  proved $C^\alpha$-regularity and Harnack's inequality for a solution to a degenerate linear elliptic 
equation.
In the above inequality, $B_r$ is a ball in $\R^n$, $1<p<\infty$, a weight $w:\R^n \to [0,\infty)$ is in the $A_p$-class, 
and $k>1$ can be chosen as 
\[
k=\frac{n}{n-1}+\delta
\]
for some $\delta>0$.
We also refer to \cite{CF85} for a simpler proof of \eqref{classical}.  After then, regularity theory for degenerate  parabolic 
equations
have been actively studied, see for instance
\cite{CS84-1,CS84,CS87,GW91,Su00}.
However, the weights  treated in the preceding papers are independent of the time variable or satisfy the $A_p$ condition on the time or the space variables separately. In this paper, we consider Sobolev spaces arising from parabolic PDEs and weights in the so-called parabolic $A_p$-class (cf. Section~\ref{Sec2}), which is a  natural Muckenhoupt class in the parabolic setting, and derive 
the Sobolev-Poincar\'e type inequalities in this parabolic setting.

More precisely, we consider a function $u\in L^1(I,W^1_1(\Omega))$ that is a  distributional solution of the following divergence type linear parabolic equations:
\begin{equation}\label{diveq}
u_t = \mathrm{div}_x\,G \quad \text{in }\ 
\Omega\times I,
\end{equation}
 where $z=(x,t)\in \Omega\times I \subset \R^n\times \R$, when $G\in L^1(\Omega\times I,\R^n)$. Here the distributional solution $u$ of \eqref{diveq} means it satisfies that
\[
\int_{\Omega\times I}u \varphi_t\, dz = \int_{\Omega\times I} G \cdot D_x\varphi \,dz  \quad \text{for all }\ \varphi\in C^\infty_0(\Omega\times I).
\]
The crucial point is that $u$ may not be differentiable with respect to $t$ but
satisfies \eqref{diveq} in the distributional sense. Regarding this point, a 
parabolic Poincar\'e type inequality for $u$ in the framework of Orlicz space, which is a larger class than the $L^p$ space, was  derived in \cite{DSSV17}.
In this paper we obtain Sobolev-Poincar\'e type inequalities for $u$ with weight $w=w(x,t)$ in the parabolic $A_p$
class and $G\in L^p_w(\Omega\times I,\R^n)$ for some $p>1$, in Theorem~\ref{Thmsolution} and its corollaries.  A typical
case of \eqref{diveq} is when $G(x,t)= \mathbf{A}(x,t)Du+ F(x,t)$, where $\mathbf{A}$ is
an $n\times n$ matrix satisfying suitable uniform ellipticity and
boundedness conditions. In this case there have been studied
regularity estimates for $Du$ when $F\in L^p_w(\Omega\times I, \R^n)$
in \cite{BPS16,Ngu17} and hence, applying our main result, 
improved regularity  for the solution  $u$ can be also observed, see Remark~\ref{rmk-application}.

Non-divergence type second order linear parabolic equations have the
form
\begin{equation}\label{nondiveq}
  u_t - a_{ij}(x,t) u_{x_ix_j} = f \quad \text{in }\ \Omega_T,
\end{equation}
where the $n\times n$ matrix $\mathbf{A}=[a_{ij}]$ satisfies suitable uniform ellipticity and boundedness conditions. 
For this equation, 
it is well known that if $\Omega$ is in $C^{1,1}$, $\mathbf A$ is continuous and $f\in L^p(\Omega_T)$, then there exists a strong solution $u$ in $W^{2,1}_p(\Omega_T)$, i.e., weak derivatives $u_t$, $Du$, $D^2u$ exist in $L^p$-space, with $u=0$ on the parabolic boundary of $\Omega_T$. (We may consider discontinuous coefficient matrix $\mathbf A$ in VMO type spaces, see \cite{KK07}.)
Here, $Du=D_xu$ and $D^2u=D^2_xu$ are the spatial gradient and Hessian of $u$, respectively.      
Regarding to the $W^{2,1}_p$-space, 
we have 
the following  Sobolev-Poincar\'e type inequality 
\begin{equation}\label{SPineq}
\left( \fint_{C_r}|Du-(Du)_{C_r}|^{pk}\,dz\right)^{\frac{1}{pk}} \le c\,r \left( \fint_{C_r}\big[|u_t|^p+|D^2u|^p\big]\, dz\right)^{\frac1p}
\end{equation} 
 for any parabolic cylinder $C_r$ and any function
$u\in W^{2,1}_p(C_r)$, where $k=\frac{n+2}{n+2-p}>1$ if $p<n+2$, see for instance 
\cite[Theorem 19]{Lie03}.  The main point in the above inequality is
that the right-hand side does not involve the derivative of $Du$ with respect to variable $t$ but $u_t$. Therefore it is enough to consider $u_t$ and $D^2u$ 
when
$L^p$-regularity theory is studied for \eqref{nondiveq}. In recent papers \cite{BL15,BLO18,DK18}
there have been studied strong solutions in the setting of weighted spaces, but
the weighted version of \eqref{SPineq} has not been investigated. In this paper, we derive the weighted version of \eqref{SPineq} in Theorem~\ref{Thmint} as a direct consequence of the main result.
Moreover, we also obtain its boundary version on an upper half region of a parabolic cube/cylinder for functions with zero value on the flat boundary.

The main difficulty is that, in contrast
to the unweighted case, we cannot take advantage of the classical
weighted Sobolev-Poincar\'e inequality \eqref{classical} on each time
slice, because  the restriction of a weight function in
the parabolic $A_p$-class on each time slice, $x\mapsto w(x,t)$, does not belong to the
$A_p$-class in general.  In order to overcome this difficulty, we
present an alternative approach. Its main idea is that 
 by making use of 
the parabolic Poincar\'e type estimates obtained in \cite{DSSV17} (cf.  Lemma~\ref{lempoinu}) we derive pointwise
estimates in terms of a parabolic version of Riesz potential, which is
called the \textit{caloric Riesz potential} (cf. \eqref{caloricRiesz})
 introduced in \cite{DM11} (see also \cite{KM14}),  
 and then we apply the boundedness of a maximal operator in $L^p_w$-spaces with the $A_p$ weight $w$ in the parabolic setting.
We notice that similar arguments have been used in, for instance,  \cite{FPW98, SW92}, where weighted Sobolev-Poincar\'e type inequalities are  derived in the setting of spaces of homogeneous type. In this paper, we are interested in solutions of parabolic PDEs and Sobolev-Poincar\'e type inequalities in relevant weighted spaces.

\section{\bf Preliminaries}\label{Sec2}

\subsection*{Notation}

We write $z=(x,t)=(x_1, \dots, x_n, t)$ 
as a point in $\R^n \times \R=\R^{n+1}$. 
We define the parabolic distance
$d_{\text{p}}$ between two points $z=(x,t)$ and $\tz=(\tilde{x},
\tilde{t})$ in $\R^{n+1}$ by
\[
d_{\text{p}}(z,\tilde z) :=\max\Big\{|x-\tilde x|, \sqrt{|t-\tilde t|}\Big\},
\]
where $|\cdot|$ is the Euclidean distance. For $\alpha>0$,  $z=(x,t)\in \R^{n+1}$ and $\mathbf r := (\tilde{\mathbf r}, r_{n+1})= (r_1,\dots,r_n,r_{n+1})$ with $r_i>0$, $i=1,\dots,n+1$, the $\alpha$-parabolic rectangle $Q^\alpha_{\mathbf r}(z)$ is defined by
$$
Q^\alpha_{\mathbf r}(z) := K_{\tilde{\mathbf r}}(x) \times ( t-\alpha r_{n+1}^2, t+\alpha r_{n+1}^2) 
$$
where $K_{\tilde{\mathbf r}}(x) := \{y=(y_1, \dots,y_n)\in \R^n: |x_i-y_i|< r_i,\  i=1,\dots, n \}$ is the $n$-dimensional rectangle.
In particular, for $r>0$, if $r_i=r$ for all $i=1,\dots,n+1$, we write 
$Q^\alpha_{r}(z)=Q^\alpha_{\mathbf r}(z)$. 
Note that $Q^\alpha_r(z)$ is called an $\alpha$-parabolic cube. In addition, 
for $\alpha,r>0$, the $\alpha$-parabolic cylinder $C^\alpha_r(z)$ is defined by
$$
C^\alpha_r(z) := B_r(x) \times ( t-\alpha r^2, t+\alpha r^2),
$$
where $B_r(x)= \{ y \in \R^{n} : |x-y|< r\}$ is the $n$-dimensional ball centered at $x$ with radius $r$.  In particular, when $\alpha=1$, we abbreviate
$Q^1_r(z)=Q_r(z)$ and $C^1_r(z)=C_r(z)$, which are called a parabolic cube and a parabolic cylinder, respectively. 
 
We also denote 
\begin{equation}\label{upperhalf}
Q_r^+(z)= Q_r(z) \cap \{ x_n>0 \} \textrm{ and }  C_r^+(z)= C_r(z) \cap \{ t>0 \}.
\end{equation}
 For the sake of simplicity, we omit the center $z$ when $z=0$ in the above notation, for example, $Q_r=Q_r(0)$, $Q_r^+=Q_r^+(0)$,\dots. In addition, we denote 
\begin{equation}\label{Tr}
 T_r = Q_r(0) \cap \left\{x_n = 0\right\}=C_r(0) \cap \left\{x_n = 0\right\}.
\end{equation} 

Let $U$ be a bounded open set in $\R^{n+1}.$ For a function $v: U  \rightarrow \R$, we denote the spatial gradient of $v$ by $Dv=(v_{x_1},\dots,v_{x_{n}})$,
the spatial Hessian of $v$ by $D^{2}v$,
and the time derivative of $v$ by $v_t$.
For an integrable function $g$ in 
$U$, we define the mean of $g$ in $U$ by 
\[
(g)_{U}:=\fint_U g(z) dz:=\frac{1}{|U|}\int_U g(z)\, dz.
\]
 Finally, the caloric Riesz potential of a measurable function $g$ in $\R^{n+1}$ is defined by
\begin{equation}\label{caloricRiesz}
\mathcal I_\beta g(z):= \int_{\R^{n+1}} \frac{g(\tz)}{d_{\mathrm{p}}(z,\tz)^{n+2-\beta}}\,d\tz, \quad z\in\R^n,
\end{equation}
where $0<\beta \le n+2$.

\subsection*{Parabolic $A_p$-class and weighted spaces} 
%
%
%
%
%
%

For $1<p<\infty$, 
a weight $w$ in $\R^{n+1}$, i.e., a locally integrable nonnegative function $w$  in $\R^{n+1}$,
is called a parabolic
\textit{$A_p$ weight}, denoted by $w \in A_p$, if  
$$ [w]_p:=\sup_{Q} \left( \fint_{Q} w\, dz \right)\left( \fint_{Q} w^{\frac{-1}{p-1}}\, dz\right)^{p-1} < \infty,$$
where the supremum is taken over all parabolic cubes $Q \subset \R^{n+1}$. Every parabolic $A_p$ weight has the doubling property, and monotonicity $A_{p_1} \subset A_{p_2}$ for $p_1 \leq p_2$.
We identify weight $w$ with measure
$w(E) := \int_{E} w(z)\, dz$,
for measurable sets $E \subset \R^{n+1}.$  For the properties of $A_p$ weight, we refer to, for instance, \cite{BPS16,Gra09}.    

Let $U\subset \R^{n+1}$ be bounded. Given $w \in A_p$, $L_w^p(U)$ denotes the \textit{weighted Lebesgue space} which contains all measurable functions $u$ on $U$ such that
$ \|u\|_{L_w^p(U)}:= \left( \int_{U} |u|^{p}w\, dz \right)^{1/p} < \infty$. We define the space $W_{p,w}^{2,1}(U)$ as the set of functions $u$ satisfying that $u,u_t,Du,D^2u$ exist  in $L_w^p(U)$.  When $w\equiv 1$, we write
$W_{p}^{2,1}(U)=W_{p,1}^{2,1}(U)$.  

For $\Omega\subset \R^{n}$, we define $W_{p}^{1}(\Omega)$ by the set of functions $u$ satisfying that $u,Du$ exist  in $L^p(\Omega)$. In addition, for $1\le p<\infty$ and an interval $I\subset \R$,  $u\in L^p(I,W^{1}_{p}(\Omega))$ means that  $u(\cdot,t)\in W^{1}_{p}(\Omega)$ a.e. $t\in I$ and $Du\in L^p(\Omega\times I)$.

\subsection*{Poincar\'e type inequalities}

We introduce Poincar\'e type inequalities under a parabolic equation in divergence form. 
 From now on,  if $a\le r_i \le b$ for all $i=1,2,\dots,n+1$ and for some $0<a<b$, we write  $a\le \mathbf r\le b$ where $\mathbf r=(r_1,\dots,r_{n+1})$.

\begin{lemma}[Theorem 2.8, \cite{DSSV17}]\label{lempoinu}
For $\alpha,r>0$, let $C=\Omega\times I \subset \R^n\times \R$ be an $\alpha$-parabolic cylinder $C^\alpha_r$ or $\alpha$-parabolic rectangle $Q^\alpha_{\mathbf r}$ with $r/2\le \mathbf r \le 2r$.  If $u\in L^1(I, W^1_1(\Omega))$ is a distributional solution of $u_t =\mathrm{div}\, G$ in $C$ 
with $G\in L^1(C,\R^n)$,
then we have
\begin{equation}\label{parapoinu}
 \fint_{C}|u-(u)_{C}|\,dz \le c\,r\fint_{C}|Du|\, dz+c \alpha r\fint_{C} |G|\, dz
\end{equation}
for some $c=c(n)>0$.
\end{lemma}

Note that in \cite[Theorem 2.8]{DSSV17}, 
only 
the $\alpha$-parabolic cylinders $C^\alpha_r$ are 
considered, but 
the inequality \eqref{parapoinu} still holds for the $\alpha$-parabolic rectangles $Q^\alpha_{\mathbf r}$ as stated in the above lemma without any significant modification to the proof.



\section{\bf Weighted Sobolev inequalities for solutions of parabolic equations in divergence form.}\label{sec4}

We consider a distributional solution $u$ to \eqref{diveq}, 
for which we obtain the following Sobolev-Poincar\'e type inequality
that is  the main result in this paper.

  

\begin{theorem}\label{Thmsolution}
Let $1<p<\infty$ and $w\in A_p$. If $u\in L^1((-r^2,r^2), W^1_1(K_r))$ with  $Du\in L^p_w(Q_r,\R^n)$ is a distributional solution to $u_t =\mathrm{div}\, G$ in $Q_r$, where $G\in L^p_w(Q_r,\R^n)$,
then $u\in L^{pk}_w(Q_r)$
with 
\begin{equation}\label{constk}
k=\frac{n+2}{n+1}+\delta>1 \quad \text{for some }\ \delta=\delta(n,p,[w]_p)>0,
\end{equation} 
 and we have
\[
\left( \frac{1}{w(Q_r)}\int_{Q_r}|u-(u)_{Q_r}|^{pk}w\,dz\right)^{\frac{1}{pk}} \le c\,r\left(\frac{1}{w(Q_r)}\int_{Q_r}\big[|Du|+|G|\big]^pw\, dz\right)^{\frac{1}{p}}
\]
for some $c=c(n,p,[w]_p)>0$. 

Moreover, the theorem holds true  if  we replace 
the $n$-dimensional cube $K_r$ and the parabolic cube $Q_r$ by the $n$-dimensional ball $B_r$ and the parabolic cylinder $C_r$, respectively.
\end{theorem}

\begin{remark}
A more specific expression of the constant $k$ can be found in \eqref{delta} in the proof of Lemma~\ref{lemmaximal} below. In the unweighted case (i.e., $w\equiv 1$) with $p<n+2$, we can choose $k=\frac{n+2}{n+2-p}$, since $w\equiv 1\in A_1$ and hence the constant $q$ in the proof of Lemma~\ref{lemmaximal} can be taken as $1$.
\end{remark}
In order to prove Theorem~\ref{Thmsolution}, we first  prove a higher integrability of the caloric Riesz potential in 
the weighted Lebesgue spaces under the parabolic setting. 
 
 \begin{lemma}\label{lemmaximal}
Let $1<p<\infty$ and $w\in A_p$. If $f\in L^p_w(\R^{n+1})$ and $f\equiv 0$ in $\R^{n+1}\setminus Q_r$, then we have
\begin{equation}\label{caloricestimate}
\left(\frac{1}{w(Q_r)}\int_{Q_r} [\mathcal I_1|f|]^{pk}w\,dz \right)^{\frac{1}{pk}}  \le c\,r \left(\frac{1}{w(Q_r)}\int_{Q_r}|f|^pw\,dz\right)^{\frac{1}{p}},
\end{equation}
where $k>1$ is given in \eqref{constk}, for some $c=c(n,p,[w]_p)>0$. 

Moreover, the lemma holds true  if we replace the parabolic cube $Q_r$ by the parabolic cylinder $C_r$.
\end{lemma}

\begin{proof}
Note that $w\in A_q$ for some $q\in (1,p)$ 
and $[w]_q>0$ depending only on $n$, $p$ and $[w]_p$, see \cite{Gra09}. Moreover, we shall assume that $q>\frac{p}{n+2}$.
For any $\epsilon>0$ we have 
$$
\mathcal I_1|f|(z) = \int_{d_{\text{p}}(\tz,z)\leq \epsilon} \frac{|f(\tz)|}{d_{\text{p}}(z,\tz)^{n+1}} \,d\tz + \int_{d_{\text{p}}(\tz,z)>\epsilon} \frac{|f(\tz)|}{d_{\text{p}}(z,\tz)^{n+1}}\,d\tz =: I_1 +I_2.
$$
For $I_1$,
$$\begin{aligned}
I_1 & = \sum_{j=0}^\infty  \int_{2^{-j-1}\epsilon< d_{\text{p}}(\tz,z)\leq 2^{-j}\epsilon} \frac{|f(\tz)|}{d_{\text{p}}(z,\tz)^{n+1}} \,d\tz 
\leq  \sum_{j=0}^\infty  2^{(j+1)(n+1)}\epsilon^{-(n+1)}\int_{d_{\text{p}}(\tz,z)\leq 2^{-j}\epsilon} |f(\tz)| \,d\tz\\
&\leq   c(n) \epsilon \sum_{j=0}^\infty  2^{-j}\fint_{Q_{2^{-j}\epsilon}(z)} |f(\tz)| \,d\tz  \leq  c(n)  \epsilon \mathcal{M} f(z)  \sum_{j=0}^\infty  2^{-j} \leq c(n) \epsilon \mathcal{M} f(z),
\end{aligned}$$
where $\mathcal M f (z):=\sup_{s>0} \fint_{Q_s(z)} |f(\tz)|\,d\tz$ is
the  parabolic maximal
function of $f$, and we used the fact that $|Q_{2^{-j}\epsilon}(z)| = 2^{n+1}(2^{-j}\epsilon)^{n+2}$.  For $I_2$, by H\"older's inequality and the facts that $f\equiv 0$ on $(Q_r)^{\mathrm c}$ and $w\in A_q$, 
\[
\begin{aligned}
I_2 & \leq \left(\int_{Q_r}|f|^p w\,d\tz\right)^{\frac{1}{p}} \left(\int_{Q_r} w^{-\frac{1}{q-1}}\,d\tz\right)^{\frac{q-1}{p}} 
\left(\int_{d_{\text{p}}(\tz,z)>\epsilon}d_{\text{p}}(z,\tz)^{-\frac{p(n+1)}{p-q}}\,dz\right)^{\frac{p-q}{p}}\\
&\leq c \left(\int_{Q_r}|f|^p w\,d\tz\right)^{\frac{1}{p}} \left(\int_{Q_r} w^{-\frac{1}{q-1}}\,d\tz\right)^{\frac{q-1}{p}} 
\epsilon^{1-\frac{q(n+2)}{p}}. 
\end{aligned}
\]
Note that 
in the last inequality we used the following fact  
$$\begin{aligned}
& \int_{d_{\text{p}}(\tz,z)>\epsilon}d_{\text{p}}(z,\tz)^{-\frac{p(n+1)}{p-q}}\,dz\\
&\qquad = 2\int^{\epsilon^2}_0\int_{\{x\in\R^n:|x|> \epsilon\}} |x|^{-\frac{p(n+1)}{p-q}}\,dx\,dt +2\int^\infty_{\epsilon^2}\int_{\{x\in\R^n: |x|> \sqrt{t}\}}|x|^{-\frac{p(n+1)}{p-q}}\,dx\,dt  \\
&\qquad \quad+2\int^\infty_{\epsilon^2}\int_{\{x\in \R^n: 0<|x|\leq  \sqrt{t}\}} 
t^{-\frac{p(n+1)}{2(p-q)}} \,dx\,dt\\
&\qquad \le c \int^{\epsilon^2}_0 \epsilon^{-\frac{p(n+1)}{p-q}+n}  \,dt + c \int^{\epsilon^2}_0 t^{-\frac{p(n+1)}{2(p-q)}+\frac{n}{2}}  \,dt = 
c \epsilon^{-\frac{p(n+1)}{p-q}+n+2}.
\end{aligned}$$

Now, we take $\epsilon>0$ such that
$$
\epsilon \mathcal{M} f(z)=\left(\int_{Q_r}|f|^p w\,d\tz\right)^{\frac{1}{p}} \left(\int_{Q_r} w^{-\frac{1}{q-1}}\,d\tz\right)^{\frac{q-1}{p}} 
\epsilon^{1-\frac{q(n+2)}{p}},
$$ 
hence 
$$
\epsilon =[\mathcal{M} f(z)]^{-\frac{p}{q(n+2)}}\left(\int_{Q_r}|f|^p w\,d\tz\right)^{\frac{1}{q(n+2)}} \left(\int_{Q_r} w^{-\frac{1}{q-1}}\,d\tz\right)^{\frac{q-1}{q(n+2)}}.
$$
Therefore, we obtain
$$
\mathcal I_1|f|(z) = I_1+I_2 \le c  [\mathcal{M} f(z)]^{1-\frac{p}{q(n+2)}}\left(\int_{Q_r}|f|^p w\,d\tz\right)^{\frac{1}{q(n+2)}} \left(\int_{Q_r} w^{-\frac{1}{q-1}}\,d\tz\right)^{\frac{q-1}{q(n+2)}}.
$$
Then, letting 
\begin{equation}\label{delta}
k=\frac{q(n+2)}{q(n+2)-p}=\frac{n+2}{n+2-p/q}=:\frac{n+2}{n+1} +\delta,
\end{equation}
we have
$$
\left(\int_{Q_r} [\mathcal I_1|f|]^{pk} w\, dz \right)^{\frac{1}{pk}} \le c \left(\int_{Q_r}|f|^p w\,d\tz\right)^{\frac{1}{q(n+2)}} \left(\int_{Q_r} w^{-\frac{1}{q-1}}\,d\tz\right)^{\frac{q-1}{q(n+2)}} \left( \int_{Q_r} [\mathcal{M} f]^p w\,dz  \right)^{\frac{1}{pk}}.
$$
Since $w\in A_p$, by the boundedness of the maximal operator in $L^p_w(\R^{n+1})$ in the parabolic setting, see \cite[Eq. (2.4)]{BPS16} and \cite{Gra09}, and using the fact that $f\equiv 0$ in $(Q_r)^{\mathrm c}$, 
$$
\left(
\int_{Q_r} [\mathcal I_1|f|]^{pk} w\, dz \right)^{\frac{1}{pk}} \le c  \left(\int_{Q_r}|f|^p w\,dz\right)^{\frac{1}{p}} \left(\int_{Q_r} w^{-\frac{1}{q-1}}\,dz\right)^{\frac{q-1}{q(n+2)}}. 
$$
Finally, since $w\in A_q$, 
$$\begin{aligned}
\left(
\int_{Q_r} [\mathcal I_1|f|]^{pk} w\, dz \right)^{\frac{1}{pk}} & \le  c \left(\int_{Q_r}|f|^p w\,dz\right)^{\frac{1}{p}} |Q_r|^{\frac{q-1}{q(n+2)}}\left(\frac{w(Q_r)}{|Q_r|}\right)^{-\frac{1}{q(n+2)}}\\
& \le c\,r \left(\frac{1}{w(Q_r)}\int_{Q_r}|f|^p w\,dz\right)^{\frac{1}{p}} w(Q_r)^{\frac{1}{pk}}. 
\end{aligned} $$
This implies \eqref{caloricestimate}.

If we consider the parabolic cylinder $C_r$ instead of $Q_r$, the result directly follows  from  \eqref{caloricestimate} and the fact that
\begin{equation}\label{wQQ+}
1\le \frac{w(Q_r)}{w(A)}\le [w]_p \left(\frac{|Q_r|}{|A|} \right)^p, \quad \text{whenever }\ A\subset Q_r
\end{equation}
(see \cite[Eq. (9.2.1)]{Gra09}). Indeed, since $C_r\subset Q_r$ and $f\equiv 0$ in $\R^{n+1}\setminus C_r$,
\[\begin{aligned}
\left(\frac{1}{w(C_r)}\int_{C_r} [\mathcal I_1|f|]^{pk}w\,dz \right)^{\frac{1}{pk}} & \leq c\left(\frac{1}{w(Q_r)}\int_{Q_r} [\mathcal I_1|f|]^{pk}w\,dz \right)^{\frac{1}{pk}}  \\
&\le c\,r \left(\frac{1}{w(Q_r)}\int_{Q_r}|f|^pw\,dz\right)^{\frac{1}{p}} \le  c\,r \left(\frac{1}{w(C_r)}\int_{C_r}|f|^pw\,dz\right)^{\frac{1}{p}}.
\end{aligned}\]

\end{proof}

Now we start the proof of Theorem \ref{Thmsolution}.

\begin{proof}[Proof of Theorem \ref{Thmsolution}]\ \\
\indent We first consider the parabolic cube $Q_r$. It suffices to prove the theorem for the case that $(u)_{Q_r}=0$, since $v:=u-(u)_{Q_r}$ also satisfies $v_t=\mathrm{div}\,G$ in the distributional sense.

Let  $\tz\in Q_r$ satisfy that
\begin{equation}\label{pf0}
\lim_{\rho\to 0}\fint_{Q_{
\rho}(
\tilde z)} |u(z)-u(\tilde z)|\,dz =0.
\end{equation}
Such a point $\tz$ is called the parabolic Lebesgue point of $u$. 
One can easily see that the set of points that are not  the parabolic Lebesgue points  has Lebesgue measure zero in $\R^{n+1}$. Define $r_j:= 2^{1-j}r$, $j=0,1,2,\dots$, and  $\tilde Q_{j} :=Q_{r_j}(\tilde z)\cap Q_r$. Then we see that $\tilde Q_{0}=Q_r$ and $\tilde Q_{j}$ is a parabolic rectangle  $Q^1_{\mathbf r}(\xi)$ with $r_j/2\le \mathbf r\le r_j$ for some $\xi \in Q_r$, hence $|\tilde Q_{j}| \approx |Q_{r_j}(\tz)| \approx r_j^{n+2}$, where relevant constants depend only on $n$.

Applying \eqref{parapoinu}  in Lemma~\ref{lempoinu} with $C=\tilde Q_j$ (i.e., $\alpha=1$),  we have 
\begin{equation}\label{pf1}\begin{aligned}
|u(\tz)|  \leq \sum_{j=0}^\infty |(u)_{\tilde Q_{j}}-(u)_{\tilde Q_{j+1}}| \leq c \sum_{j=0}^\infty \fint_{\tilde Q_{j}} |u-(u)_{\tilde Q_{j}}|\,dz  \le c\sum_{j=0}^\infty \fint_{\tilde Q_{j}} r_{j} \left[|Du|+|G|\right]\,dz.
 \end{aligned}\end{equation}
Then, since $Q_{r_j}(\tz)=\ \cdot  \hspace{-0.3cm}\bigcup_{i=j}^\infty  (Q_{r_i}(\tz)\setminus Q_{r_{i+1}}(\tz))$, for almost every $\tz \in Q_r$,
 $$\begin{aligned}
|u(\tz)|& \le c \sum_{j=0}^\infty \sum_{i=j}^\infty \int_{(Q_{r_i}(\tz)\setminus Q_{r_{i+1}}(\tz))\cap Q_r} r_{j}^{-n-1} [|Du|+|G|]\,dz  \\
& =  c \sum_{i=0}^\infty \bigg(\sum_{j=0}^i r_j^{-n-1}\bigg) \int_{(Q_{r_i}(\tz)\setminus Q_{r_{i+1}}(\tz))\cap Q_r}  [|Du|+|G|] \,dz\\
& = c \sum_{i=0}^\infty  \bigg(\sum_{j=0}^i 2^{-(n+1)(i-j)}\bigg) \int_{(Q_{r_i}(\tz)\setminus Q_{r_{i+1}}(\tz))\cap Q_r}r_i^{-n-1}  [|Du|+|G|] \,dz\\
&\le c \sum_{i=0}^\infty \int_{(Q_{r_i}(\tz)\setminus Q_{r_{i+1}}(\tz))\cap Q_r} r_i^{-n-1} [|Du|+|G|] \,dz\\
& \le c  \int_{Q_r}   \frac{|Du(z)|+|G(z)|}{d_{\text{p}}(z,\tz)^{n+1}}\,dz =c \mathcal I_1(|Du|+|G|)\chi_{Q_r} (\tz).
\end{aligned}$$
Finally, applying Lemma~\ref{lemmaximal} with $f=(|Du|+|G|)\chi_{Q_r}$, we have
\[\begin{aligned}
\left(\frac{1}{w(Q_r)}\int_{Q_r} |u|^{pk}w\,dz \right)^{\frac{1}{pk}} & \le \left(\frac{1}{w(Q_r)}\int_{Q_r} \big[\mathcal I_1(|Du|+|G|)\chi_{Q_r}\big]^{pk}w\,dz \right)^{\frac{1}{pk}}  \\
&\le c\,r \left(\frac{1}{w(Q_r)}\int_{Q_r}\big[|Du|+|G|\big]^pw\,dz\right)^{\frac{1}{p}}.
\end{aligned}\]
This completes the proof.

We next consider the parabolic cylinder $C_r$. Let  $\tz=(\tilde x,\tilde t)\in C_r$ satisfy that 
\[
\lim_{\rho\to 0}\fint_{C_{
\rho}(
\tilde z)} |u(z)-u(\tilde z)|\,dz =0.
\] 
Note that this equality is equivalent to \eqref{pf0}.
Define $r_j:= 2^{1-j}r$, $j=0,1,2,\dots$. Then there exists $j_0\in \mathbb{N}$ such that $C_{r_{j_0+1}}(\tilde z)\subset C_r \not\subset C_{r_{j_0}}(\tilde z)$. Put $\tilde C_{0}=C_r$ and $\tilde C_{j}=C_{r_j}(\tz)$ for $j >j_0$. On the other hand, for $j=1,\dots,j_0$, 
we can find $\rho_j,\alpha_j>0$ and $z_j=(y_j,\tau_j)\in C_{r_j}(\tz)\cap C_r$ such that
\[
\tfrac{1}{2}r_j\le \rho_j \le r_j, \ \ \tfrac{1}{2}\le \alpha_j\le 4,\ \ \text{and}\ \ \tilde C_j :=C^{\alpha_j}_{\rho_j}(z_j)\subset C_{r_j}(\tz)\cap C_r.
\]
Indeed, we can choose
\[
\rho_j:= 
\left\{\begin{array}{ccl} 
\frac{r_j+ r-|\tilde x|}{2} & \text{if} &  r_j > r-|\tilde x|, \\
r_j & \text{if} &  r_j \le r-|\tilde x|, 
\end{array}\right.
\quad 
\alpha_j:= 
\left\{\begin{array}{ccl} 
\tfrac{r_j^2+r^2-|\tilde t|}{2\rho_j^2} & \text{if} &  r_j^2 > r^2-|\tilde t|, \\
\tfrac{r_j^2}{\rho_j^2} & \text{if} &  r_j^2 \le r^2-|\tilde t |,
\end{array}\right.
\]
\[
y_j:= 
\left\{\begin{array}{ccl} 
(r-\rho_j)\frac{\tilde x}{|\tilde x|} & \text{if} &  r_j > r-|\tilde x|, \\
\tilde x & \text{if} &  r_j \le r-|x|, 
\end{array}\right.
\quad \text{and}\quad
\tau_j:= 
\left\{\begin{array}{ccl} 
(r^2-\alpha_j\rho_j^2)\tfrac{\tilde t}{|\tilde t|} & \text{if} &  r_j^2 > r^2 -|\tilde t|, \\
\tilde t & \text{if} &  r_j^2 \le r^2-|\tilde t|,
\end{array}\right.
\]
so that $B_{\rho_j}(y_j)$ is the largest ball in $B_{r_j}(\tilde x)\cap B_r$ and $(\tau_j-\alpha_j \rho_j^2,\tau_j+\alpha_j \rho_j^2)= (\tilde t- r_j^2,\tilde t + r_j^2)\cap (- r^2,r^2)$.
Note that for every $j=0,1,2,\dots$, $\tilde C_{j+1} \subset \tilde C_{j} \subset  \tilde C_0 =C_r$ and  $|\tilde C_{j}| \approx r_j^{n+2}$ with relevant constant depending only on $n$ and independent of $j$.
Therefore, applying \eqref{parapoinu} in Lemma~\ref{lempoinu}  with $C=\tilde C_j$, we have 
\[\begin{aligned}
|u(\tz)|  \leq c \sum_{j=0}^\infty \fint_{\tilde C_{j}} |u-(u)_{\tilde C_{j}}|\,dz & \le c\sum_{j=0}^\infty \fint_{\tilde C_{j}} r_{j} \left[|Du|+|G|\right]\,dz \\
& \le c\sum_{j=0}^\infty \fint_{C_{r_j}(\tz)\cap C_r} r_{j} \left[|Du|+|G|\right]\,dz, 
\end{aligned}\]
which is the counterpart of \eqref{pf1}. The rest part of the proof is exactly same as the one for the parabolic 
cube $Q_r$. 
\end{proof}

The following higher integrability of $|u|^p w$ follows from Theorem~\ref{Thmsolution}.
\begin{corollary}\label{Cor1}
Let $1<p<\infty$ and $w\in A_p$. If $u\in L^1((-r^2,r^2), W^1_1(K_r))$ with  $Du\in L^p_w(Q_r,\R^n)$ is a distributional solution to $u_t =\mathrm{div}\, G$ in $Q_r$, where $G\in L^p_w(Q_r,\R^n)$,
 then $|u|^pw \in L^{\gamma}(Q_r)$ for some $\gamma \in (1,k)$ depending only on $n$, $p$ and $[w]_p$, where $k>1$ is given in \eqref{constk} (see \eqref{gamma} below), and we have
$$\begin{aligned}
&\left(\frac{1}{|Q_r|}\int_{Q_r} \big[|u-(u)_{Q_r}|^{p}w\big]^{\gamma}\,dz \right)^{\frac{1}{p\gamma}}   \le c \, r \left(\frac{1}{|Q_r|}\int_{Q_r}\left[|Du|^p+|G|^p\right]w\,dz\right)^{\frac{1}{p}}
\end{aligned}$$
for some $c=c(n,p,[w]_p)>0$. 

Moreover, the corollary holds true if  we replace the parabolic cube $Q_r$ by the parabolic cylinder $C_r$.
\end{corollary} 
%
%
%
%
%

\begin{proof} We recall the following reverse H\"older type inequality for an $A_p$-weight:
\begin{equation}\label{reverse}
\left(\fint_{Q_r} w^{1+\epsilon_0}\,dz\right)^{\frac{1}{1+\epsilon_0}}\le c \fint_{Q_r} w\,dz,
\end{equation}
where $\epsilon_0,c>0$ depend only on $n$, $p$ and $[w]_p$ (see \cite{Gra09}). Then choose $\gamma\in(1,k)$ such that
\begin{equation}\label{gamma}
1< \frac{(k-1)\gamma}{k-\gamma} < 1+\epsilon_0.
\end{equation}
 By H\"older's inequality, 
$$\begin{aligned}
\frac{1}{w(Q_r)}\int_{Q_r} &\big[|u-(u)_{Q_r}|^{p}w\big]^{\gamma}\,dz  = \frac{1}{w(Q_r)}\int_{Q_r} |u-(u)_{Q_r}|^{p\gamma}w^{\frac{\gamma}{k}}w^{\gamma-\frac{\gamma}{k}}\,dz \\
& \quad \leq \left(\frac{1}{w(Q_r)}\int_{Q_r} |u-(u)_{Q_r}|^{pk}w\,dz \right)^{\frac{\gamma}{k}}\left(\frac{1}{w(Q_r)}\int_{Q_r} w^{\frac{(k-1)\gamma}{k-\gamma}}\,dz\right)^{\frac{k-\gamma}{k}},
\end{aligned}$$
and moreover, using H\"older's inequality with \eqref{gamma} and \eqref{reverse},
$$\begin{aligned}
\left(\frac{1}{w(Q_r)}\int_{Q_r} w^{\frac{(k-1)\gamma}{k-\gamma}}\,dz\right)^{\frac{k-\gamma}{k}} & =  \left(\frac{|Q_r|}{w(Q_r)}\right)^{\frac{k-\gamma}{k}}\left(\fint_{Q_r} w^{\frac{(k-1)\gamma}{k-\gamma}}\,dz\right)^{\frac{k-\gamma}{k}}\\
&\le c \left(\frac{|Q_r|}{w(Q_r)}\right)^{\frac{k-\gamma}{k}}\left(\fint_{Q_r} w\,dz\right)^{\frac{(k-1)\gamma}{k}} =c\left(\frac{w(Q_r)}{|Q_r|}\right)^{\gamma-1}.
\end{aligned}$$
Therefore, it directly follows from Theorem~\ref{Thmsolution} that
$$\begin{aligned}
&\left(\frac{1}{w(Q_r)}\int_{Q_r}\big[ |u-(u)_{Q_r}|^{p}w\big]^{\gamma}\,dz \right)^{\frac{1}{p\gamma }}\\
&\qquad \le c\,  r \left(\frac{1}{w(Q_r)}\int_{Q_r}\left[|Du|^p+|G|^p\right]w\,dz\right)^{\frac{1}{p}} \left(\frac{w(Q_r)}{|Q_r|}\right)^{\frac{\gamma-1}{\gamma p}},
\end{aligned}$$
which implies the desired estimates.
\end{proof}

%

We next consider 
distributional solutions to
\eqref{diveq} in the upper half parabolic cube $Q_r^+$, or cylinder $C^+_r$, in \eqref{upperhalf} with zero boundary condition on $T_r$ in \eqref{Tr}.

\begin{theorem}\label{Thmsolutionhalf}
Let $1<p<\infty$ and $w\in A_p$. If $u\in L^1((-r^2,r^2), W^1_1(K_r^+))$ with  $Du\in L^p_w(Q_r,\R^n)$ is a distributional solution of $u_t =\mathrm{div}\, G$ in $Q_r^+$  with $u=0$ on $T_r$, where $G\in L^p_w(Q_r^+,\R^n)$,
then $u\in L^{pk}_w(Q_r^+)$ with $k>1$ given in \eqref{constk}, and we have
\begin{equation}\label{sobopoinbd}
\left( \frac{1}{w(Q_r^+)}\int_{Q^+_r}|u|^{pk}w\,dz\right)^{\frac{1}{pk}} \le c\,r\left(\frac{1}{w(Q_r^+)}\int_{Q^+_r}\big[|Du|+|G|\big]^pw\, dz\right)^{\frac{1}{p}}
\end{equation}
for some $c=c(n,p,[w]_p)>0$. 

Moreover, the theorem holds true if  we replace the parabolic cube $Q_r$ by the parabolic cylinder $C_r$.

\end{theorem}

%

\begin{proof}
This is a consequence of Theorem \ref{Thmsolution} with an extension argument. For a function $v$ on $Q_r^+$ we define $\tilde v$ and  $\overline v$ on $Q_r$ by the odd extension and the even extension of $v$, respectively, i.e.,
\[
\tilde v(z)=v(x_1,\dots,x_n,t) =
\left\{
\begin{array}{ccl}
v(x_1,\dots,x_{n-1},x_n,t) & \text{if} & x_n > 0,\\
-v(x_1,\dots,x_{n-1},-x_n,t) & \text{if} & x_n< 0,
\end{array}
\right.
\]
and
\[
\overline v(z)=v(x_1,\dots,x_n,t) =
\left\{
\begin{array}{ccl}
v(x_1,\dots,x_{n-1},x_n,t) & \text{if} & x_n > 0,\\
v(x_1,\dots,x_{n-1},-x_n,t) & \text{if} & x_n< 0.
\end{array}
\right.
\]
Let $G=(g_1,\dots,g_n)$ and define $G^*:=(\widetilde{g_1},\dots,\widetilde{g_{n-1}},\overline{g_n})$. Then we can see that  $\tilde u$ (the even extension of $u$)  is in $L^1((-r^2,r^2), W^1_1(K_r))$ and satisfies 
\begin{equation}\label{tildeu}
\tilde u_t =\mathrm{div}\, G^*\ \  \text{in } Q_r, \quad \text{in the distribution sense.}
\end{equation}
This can be proved by using a cut-off function, see for instance \cite[Theorem 3.4]{Mart81}. Indeed, for any small $\epsilon\in(0,1)$, let $U_\epsilon=\{(x',x_n,t): 0\le x_n \le \epsilon\}$ and $\eta_\epsilon\in C^\infty_0(\R)$ such that $0\le \eta_\epsilon\le 1$, $\eta_\epsilon(\tau)=1$ if $|\tau|\le \frac{\epsilon}{2}$, $\eta_\epsilon(\tau)=0$ if $|\tau|\ge \epsilon$,  $\eta_\epsilon(-\tau)=\eta_\epsilon(\tau)$ for all $\tau\in\R$, and  $|\eta_\epsilon'|\le 4/\epsilon$.
Then we have that for any $\phi\in C^\infty_0(Q_r)$,
$$\begin{aligned}
\int_{Q_r} \tilde u \phi_t\,dz &= \int_{Q_r^+} \tilde u (\phi(x',x_n,t)-\phi(x',-x_n,t))_t\,dz\\
&= \int_{Q_r^+}  u [(1-\eta_\epsilon(x_
n))(\phi(x',x_n,t)-\phi(x',-x_n,t)]_t\,dz \\
&\qquad +\underbrace{\int_{Q_r^+\cap U_\epsilon}  u [\eta_\epsilon(x_n)(\phi(x',x_n,t)-\phi(x',-x_n,t)]_t\,dz}_{=:I_1}\\
&= \int_{Q_r^+} G \cdot D [(1-\eta_\epsilon(x_
n))(\phi(x',x_n,t)-\phi(x',-x_n,t))]\,dz + I_1\\
&= \int_{Q_r^+} G \cdot D (\phi(x',x_n,t)-\phi(x',-x_n,t))\,dz \\
&\qquad + \underbrace{\int_{Q_r^+\cap U_\epsilon} G \cdot D [\eta_\epsilon(x_
n)(\phi(x',x_n,t)-\phi(x',-x_n,t))]\,dz}_{=:I_2} + I_1\\
 & = \int_{Q_r} G^* \cdot D \phi\,dz+ I_1+I_2. 
\end{aligned}$$
This implies \eqref{tildeu} since 
$$
|I_1| \le  c \|\phi_t\|_{\infty} \int_{Q^+_r\cap U_\epsilon}|u|\,dz \ \ \longrightarrow \ \ 0 \quad \text{as }\ \epsilon \ \to \ 0 ,
$$
$$
|I_2| \le  c \|D\phi\|_{\infty} \int_{Q^+_r\cap U_\epsilon}|G|\,dz \ \ \longrightarrow \ \ 0 \quad \text{as }\ \epsilon \ \to \ 0 .
$$

Finally, from Theorem \ref{Thmsolution} with the facts that $(\tilde u)_{Q_r}=0$ and $1\le \frac{w(Q_r)}{w(Q^+_r)}\le [w]_p 2^p$ by \eqref{wQQ+}, we have that 
\[\begin{aligned}
\left( \frac{1}{w(Q_r^+)}\int_{Q^+_r}|u|^{pk}w\,dz\right)^{\frac{1}{pk}} & \le \left( \frac{1}{w(Q_r)}\int_{Q_r}|\tilde u|^{pk}w\,dz\right)^{\frac{1}{pk}} \\
& \le c\, r\left(\frac{1}{w(Q_r)}\int_{Q_r}\big[|D\tilde u|+|G^*|\big]^pw\, dz\right)^{\frac{1}{p}} \\
&\le  c\, r\left(\frac{1}{w(Q_r^+)}\int_{Q_r^+}\big[|Du|+|G|\big]^pw\, dz\right)^{\frac{1}{p}}. \qedhere
\end{aligned}\]
\end{proof}

\begin{remark} In stead of the zero boundary condition on the flat, we can consider a zero initial condition. More precisely, under the same setting as in the above theorem, if $u\in L^{1}((-r^2,r^2),W^1_1(K_r))$ with  $Du\in L^p_w(Q_r,\R^n)$ is a distributional solution of $u_t =\mathrm{div}\, G$ in $Q_r$, where $G\in L^p_w(Q_r,\R^n)$ and satisfies that 
\[
\lim_{t \searrow -r^s}\fint_{-r^s}^t \left( \int_{K_r}|u(x,s)|\, dx \right) \, ds = 0,
\]
then we have the estimate \eqref{sobopoinbd}, replacing $Q_r^+$ with $Q_r$. Its proof is almost the same as the one of the above theorem.   
\end{remark}

We end this section with an application of the above results to typical linear parabolic equations in divergence form. 
\begin{remark} \label{rmk-application}
We consider the following linear parabolic equation 
\begin{equation}\label{linearpara}
\left\{
\begin{array}{rclcl}
u_t -\mathrm{div} (\mathbf{A}(x,t)Du) & =& \mathrm{div} F & \text{in} & C_r,\\
u & =& 0&\text{on} & \partial_{\text{p}}(C_r),
\end{array}
\right.
\end{equation}
where $\partial_{\text{p}}(C_r) =( \partial B_r \times (-r^2, r^2)) \cup (B_r \times \{t = -r^2 \} ) $
 as the parabolic boundary of $C_r$. Then in view of \cite{Ngu17}, one can see that, under suitable assumptions on $\mathbf{A}$, for instance that $\mathbf{A}$ 
is bounded,
satisfies the uniform ellipticity and is of VMO(vanishing mean oscillation), if $F\in L^p_w(C_r,\R^n)$ for some $1<p<\infty$ and $w\in A_p$,  there exists $u\in L^q((-r^2,r^2); W^1_q(B_r))$ for some $q>1$ such that $u$ is a distributional solution of  \eqref{linearpara} and 
\[
\int_{C_r} |Du|^{p}w\, dz \le c \int_{C_r} |F|^{p}w\, dz.
\]
Therefore, by Theorem~\ref{Thmsolutionhalf}, we obtain that $u\in L^{pk}_w(C_r)$ and 
\[
\left( \frac{1}{w(C_r)} \int_{C_r}|u-(u)_{C_r}|^{pk}w\,dz\right)^{\frac{1}{pk}} \le c\,r\left(\frac{1}{w(C_r)}\int_{C_r}|F|^pw\, dz\right)^{\frac{1}{p}}.
\]
\end{remark}

\section{\bf Weighted Sobolev-Poincar\'e type  inequalities for the spatial gradient}

We consider functions in $W^{2,1}_1$-space and obtain weighted Sobolev-Poincar\'e type inequalities for $Du$ with a weight $w\in A_p$ and $1<p<\infty$ as consequences of the results in the previous section.

\begin{theorem}\label{Thmint} Let $1<p<\infty$ and $w\in A_p$. If $u\in W^{2,1}_{1}(Q_r)$ and $u_t, |D^2u| \in L^p_w(Q_r)$, then $Du \in L^{pk}_w(Q_r,\R^n)$ with  $k>1$ given in \eqref{constk}, and we have
$$
\left(\frac{1}{w(Q_r)}\int_{Q_r} |Du-(Du)_{Q_r}|^{pk}w\,dz \right)^{\frac{1}{pk}}  \le c\,r \left(\frac{1}{w(Q_r)}\int_{Q_r}\left[|D^2u|^p+|u_t|^p\right]w\,dz\right)^{\frac{1}{p}}
$$
for some $c=c(n,p,[w]_p)>0$.
\end{theorem}
\begin{proof}
For $i=1,2,\dots,n$, set $v_i:=u_{x_i}$. Then we have $(v_i)_{t}= (u_t)_{x_i}= \mathrm{div}\, G^i$ a.e. in $Q_r$, where $G^i=(g^i_1,\dots,g^i_n)$ with $g^i_i=u_t$ and $g^i_j=0$ for all $j\neq i$. Therefore, by Theorem~\ref{Thmsolution}, we have for each $i=1,\dots,n$,
$$\begin{aligned}
\left(\frac{1}{w(Q_r)}\int_{Q_r} |u_{x_i}-(u_{x_i})_{Q_r}|^{pk}w\,dz \right)^{\frac{1}{pk}}  &\le c\,r \left(\frac{1}{w(Q_r)}\int_{Q_r}\left[|D(u_{x_i})|^p+|G^i|^p\right]w\,dz\right)^{\frac{1}{p}}\\
& \le c\,r \left(\frac{1}{w(Q_r)}\int_{Q_r}\left[|D^2u|^p+|u_t|^p\right]w\,dz\right)^{\frac{1}{p}}.
\end{aligned}$$
This completes the proof.
\end{proof}

\begin{theorem}\label{Thmbdy}Let $1<p<\infty$ and $w\in A_p$. If $u\in W^{2,1}_{1}(Q_r^+)$ with $u_t, |D^2u| \in L^p_w(Q_r^+)$ and $u =0 $ on $T_r$, then $Du \in L^{pk}_w(Q_r^+,\R^n)$ with  $k>1$ given in \eqref{constk}, and we have
$$\begin{aligned}
&\bigg(\frac{1}{w(Q_r^+)}  \int_{Q^+_r}\bigg[ \sum_{i=1}^{n-1} |u_{x_i}|+|u_{x_n}-(u_{x_n})_{Q^+_r}| \bigg]^{pk} w\,dz \bigg)^{\frac{1}{pk}}\\
& \qquad\qquad \qquad\qquad\qquad\le c \, r \left(\frac{1}{w(Q_r
^+)}\int_{Q_r^+}\left[|u_t|^p+|D^2u|^p\right]w\,dz\right)^{\frac{1}{p}}
\end{aligned}
$$
for some $c=c(n,p,[w]_p)>0$. 
\end{theorem}
\begin{proof}
Since $u_{x_i}=0$ on $T_r$ for $i=1,\dots,n-1$, in view of the proofs of Theorem~\ref{Thmint} and  Theorem~\ref{Thmsolutionhalf} we have that for each $i=1,\dots,n-1$,
$$\begin{aligned}
\left(\frac{1}{w(Q^+_r)}\int_{Q_r} |u_{x_i}|^{pk}w\,dz \right)^{\frac{1}{pk}}   \le c\,r \left(\frac{1}{w(Q_r^+)}\int_{Q_r^+}\left[|D^2u|^p+|u_t|^p\right]w\,dz\right)^{\frac{1}{p}}.
\end{aligned}$$

On the other hand, for $u_{x_n}$, we directly apply Theorem~\ref{Thmsolution}. Note that clearly Theorem~\ref{Thmsolution} still holds  if we replace $Q_r$ by $Q_r^+$. Since $(v_i)_{t}= (u_t)_{x_i}= \mathrm{div}\, G$ a.e. in $Q_r^+$, where $G^i=(0,\dots,0,u_t)$, we obtain
$$\begin{aligned}
\left(\frac{1}{w(Q_r^+)}\int_{Q_r^+} |u_{x_n}-(u_{x_n})_{Q_r^+}|^{pk}w\,dz \right)^{\frac{1}{pk}}  &\le c\,r \left(\frac{1}{w(Q_r^+)}\int_{Q_r^+}\left[|D(u_{x_n})|^p+|G|^p\right]w\,dz\right)^{\frac{1}{p}}\\
& \le c\,r \left(\frac{1}{w(Q_r^+)}\int_{Q_r^+}\left[|D^2u|^p+|u_t|^p\right]w\,dz\right)^{\frac{1}{p}}.
\end{aligned}$$
This completes the proof.
\end{proof}

\begin{remark}
Theorems~\ref{Thmsolutionhalf}, \ref{Thmint} and \ref{Thmbdy} still hold  if we replace the parabolic cube $Q_r$ by the parabolic cylinder $C_r$. Moreover, we also have the counterparts of Corollary~\ref{Cor1} for those theorems.
\end{remark}

\section{\bf Acknowledgment}
We thank the referee for helpful comments.
L. Diening was supported by the Deutsche Forschungsgemeinschaft
(DFG, German Research Foundation) – SFB 1283/2 2021 – 317210226. 
M. Lee was supported by the
National Research Foundation of Korea (NRF-2019R1F1A1061295). J. Ok was supported by the National
Research Foundation of Korea (NRF-2017R1C1B2010328)


\bibliographystyle{amsplain}

\end{document}